\documentclass[12pt]{article}
\usepackage{amsmath, amssymb,amsthm}
\usepackage{graphicx}

\usepackage{amsthm}

\usepackage{enumerate}

\newcommand{\ex}{\mathrm{ex}}
\newcommand{\Ex}{\mathrm{Ex}}
\newcommand{\co}{\mathrm{Co}}
\newcommand{\ro}{\mathrm{Ro}}

\newtheorem{theorem}{Theorem}
\newtheorem{lemma}{Lemma}

\newtheorem{fact}{Fact}

\def\le{\leqslant}
\def\ge{\geqslant}
\usepackage{lineno}

\begin{document}

\title{Refined Tur\'an numbers and  Ramsey numbers for the loose 3-uniform path of length three}

\author{Joanna Polcyn
\\A. Mickiewicz University\\Pozna\'n, Poland\\{\tt joaska@amu.edu.pl}
 \and Andrzej Ruci\'nski\thanks{Research supported by the Polish NSC
grant 2014/15/B/ST1/01688.
}
\\A. Mickiewicz University\\Pozna\'n, Poland\\{\tt rucinski@amu.edu.pl}
}

\date{\today}

\maketitle

\begin{abstract}
Let $P$ denote  a 3-uniform hypergraph consisting of 7 vertices $a,b,c,d,e,f,g$ and 3 edges
$\{a,b,c\}, \{c,d,e\},$ and $\{e,f,g\}$. It is known that the $r$-color Ramsey number for $P$ is
 $R(P;r)=r+6$  for $r\le 7$. The proof of this result relies on a careful analysis
of the Tur\'an numbers for $P$. In this paper, we refine this analysis  further and compute, for all $n$, the third and fourth order Tur\'an numbers for $P$.
With the help of the former, we  confirm the formula
$R(P;r)=r+6$ for $r\in\{8,9\}$.
\end{abstract}

\section{Introduction}\label{intro}
For brevity, 3-uniform hypergraphs will be called here \emph{ $3$-graphs}.
 Given a family of $3$-graphs $\mathcal F$, we say that a $3$-graph $H$ is  \emph{$\mathcal F$-free}
if for all $F \in \mathcal F$ we have  $H \nsupseteq F$.

 For a family of $3$-graphs $\mathcal F$ and an integer $n\ge1$, the \textit{ Tur\'an number of the 1st order}, that is, the ordinary Tur\'an number, is defined as
    $$
    \mathrm{ex}^{(1)}(n; \mathcal F)=\max\{|E(H)|:|V(H)|=n\;\mbox{ and $H$ is
    $\mathcal F$-free}\}.
    $$
 Every $n$-vertex $\mathcal F$-free $3$-graph with $\ex^{(1)}(n;\mathcal F)$
 edges is called  \emph{1-extremal for~$\mathcal F$}.
  We denote by $\mathrm{Ex}^{(1)}(n;\mathcal F)$
  the family of all, pairwise non-isomorphic, $n$-vertex $3$-graphs which are 1-extremal for $\mathcal F$.
Further,  for an  integer $s\ge1$,
 \textit{the Tur\'an number of the $(s+1)$-st order} is defined as
\begin{eqnarray*}
\mathrm{ex}^{(s+1)}(n;\mathcal F)=\max\{|E(H)|:|V(H)|=n,\; \mbox{$H$ is
    $\mathcal F$-free, and }\\
 \forall H'\in \mathrm{Ex}^{(1)}(n;\mathcal F)\cup...\cup\mathrm{Ex}^{(s)}(n;\mathcal F),  H\nsubseteq H'\},
\end{eqnarray*}
if such a $3$-graph $H$ exists. Note that if $\mathrm{ex}^{(s+1)}(n;\mathcal F)$ exists then, by definition, 
\begin{equation}\label{decrease}
\mathrm{ex}^{(s+1)}(n;\mathcal F)<\mathrm{ex}^{(s)}(n;\mathcal F).
\end{equation}

An $n$-vertex $\mathcal F$-free $3$-graph $H$ is called \textit{(s+1)-extremal for} $\mathcal F$ if $|E(H)| = \mathrm{ex}^{(s+1)}(n;\mathcal F)$ and  $\forall H'\in \mathrm{Ex}^{(1)}(n;\mathcal F)\cup...\cup\mathrm{Ex}^{(s)}(n;\mathcal F),  H\nsubseteq H'$; we denote by $\mathrm{Ex}^{(s+1)}(n;\mathcal F)$ the family of $n$-vertex $3$-graphs which are $(s+1)$-extremal for $\mathcal F$.
 In the case when $\mathcal F=\{F\}$, we will write $F$ instead of  $\{ F\}$.

\emph{A loose 3-uniform path of length 3} is a 3-graph $P$ consisting of 7 vertices, say,
$a,b,c,d,e,f,g$, and 3 edges $\{a,b,c\}, \{c,d,e\},$ and $\{e,f,g\}$.
The \textit{Ramsey number} $R(P;r)$ is the least integer $n$ such that  every $r$-coloring
of the edges of the complete $3$-graph $K_n$ results in a monochromatic copy of $P$. Gyarfas and Raeisi \cite{GR} proved, among many other results, that $R(P;2)=8$. (This result was later extended to loose paths of arbitrary lengths, but still $r=2$, in \cite{Omidi}.)
Then Jackowska \cite{J} showed that $R(P;3)=9$  and $r+6\le R(P;r)$ for all $r\ge3$.
In turn, in \cite{JPR} and \cite{JPRr}, Tur\'an numbers of the first and second order,  $\ex^{(1)}(n;P)$ and  $\ex^{(2)}(n;P)$, have been determined for all feasible values of $n$, as well as the single third order  Tur\'an number $\ex^{(3)}(12;P)$. Using these numbers, in \cite{JPRr}, we were able to compute the Ramsey numbers $R(P;r)$ for $r=4,5,6,7$.

 \begin{theorem}[\cite{GR, J,JPRr}]\label{past}
    For all $r\le 7$, $R(P; r)=r+6$.
\end{theorem}

In this paper we determine, for all $n\ge7$, the  Tur\'an numbers for $P$ of the third and fourth order, $\ex^{(3)}(n;P)$ and $\ex^{(4)}(n;P)$. The former allows us to compute two more Ramsey numbers.

 \begin{theorem}\label{main}
    For all $r\le 9$, $R(P; r)=r+6$.
\end{theorem}

It seems that in order to make a further progress in computing the Ramsey numbers $R(P;r)$, $r\ge10$, one would need to determine  higher order Tur\'an numbers $\mathrm{ex}^{(s)}(n;P)$, at least for some small values of $n$. Unfortunately, the forth order numbers are not good enough.

Throughout, we denote by $S_n$ the 3-graph on $n$ vertices and with $\binom {n-1}2$ edges, in which one vertex, referred to as \emph{the center}, forms  edges with all  pairs of the remaining vertices.
Every sub-3-graph of $S_n$ without isolated vertices is called \emph{a star}, while $S_n$ itself is called  \emph{the full star}. We  denote by $C$ \emph{the triangle}, that is, a 3-graph with six vertices $a,b,c,d,e,f$ and three edges $\{a,b,c\}$, $\{c,d,e\}$, and $\{e,f,a\}$. Finally, $M$ stands for a pair of disjoint edges.

 In the next section we state all, known and new, results on ordinary and  higher order Tur\'an numbers for $P$, including Theorem \ref{ex3_cale} which provides a complete formula for  $\ex^{(3)}(n;P)$.  We also define conditional Tur\'an numbers and quote from  \cite{JPRr}  three useful
  lemmas  about the  conditional Tur\'an numbers  with respect to  $P$,  $C$,  $M$.
 Then, in Section \ref{proofRam},   we prove Theorem \ref{main}, while
 the remaining sections  are  devoted to proving Theorem \ref{ex3_cale}.

\section{Tur\'an numbers}\label{turan}

A  celebrated result of Erd\H os, Ko, and Rado \cite{EKR} asserts that for $n\ge 6$, $\mathrm{ex}^{(1)}(n;M)=\binom{n-1}2$. Moreover, for $n\ge7$,  $\mathrm{Ex}^{(1)}(n;M)=\{S_n\}$.
We will need the second order version of this Tur\'an number, together with the 2-extremal family. Such a result has been proved already by Hilton and Milner \cite{HM} (see \cite{FF2} for a simple proof). For a given set of vertices $V$, with $|V|=n\ge 7$, let us define two special 3-graphs. Let $x,y,z,v\in V$ be four different vertices of $V$. We set
$$
G_1(n)=\left\{\{x,y,z\}\right\}\cup\left\{h\in \binom{V}{3}: v\in h, h\cap \{x,y,z\}\neq\emptyset\right\},
$$
$$
G_2(n)=\left\{\{x,y,z\}\right\}\cup\left\{h\in \binom{V}{3}: |h\cap \{x,y,z\}|=2\right\}.
$$
Note that for  $i\in \{1,2\}$, $G_i(n)\not\supset M$ and $|G_i(n)|=3n-8$.

\begin{theorem}[\cite{HM}]\label{ntif}
	For  $n\ge 7$,  $\mathrm{ex}^{(2)}(n;M)=3n-8$ and $\Ex^{(2)}(n;M)=\{G_1(n),G_2(n)\}$.
\end{theorem}
\noindent Later, we will use the fact that   $C\subset  G_i(n)\not\supset P$, $i=1,2$.

 Recently, the third order Tur\'an number for $M$ has been established by Han and Kohayakawa. Let $G_3(n)$ be the 3-graph on $n$ vertices, with distinguished vertices $x,y_1,y_2,z_1,z_2$ whose edge set consists of all edges spanned by $x,y_1,y_2,z_1,z_2$ except for $\{y_1,y_2,z_i\}$, $i=1,2$, and all edges of the form $\{x,z_i,v\}$, $i=1,2$, where $v\not\in\{x,y_1,y_2,z_1,z_2\}$.

\begin{theorem}[\cite{HK}]\label{ntntif}
	For  $n\ge 7$,  $\mathrm{ex}^{(3)}(n;M)=2n-2$ and  $\Ex^{(3)}(n;M)=\{G_3(n)\}$.
\end{theorem}

\bigskip
Interestingly, the number $\binom{n-1}2$ serves as the Tur\'an number for two other 3-graphs, $C$ and $P$.
 The Tur\'an
number $\ex^{(1)}(n;C)$ has been determined in \cite{FF} for $n\ge 75$ and later for all $n$ in
\cite{CK}.
\begin{theorem}[\cite{CK}]\label{c3}
    For $n\ge 6$, $\ex^{(1)}(n;C)=\binom{n-1}{2}$. Moreover, for $n\ge8$, \newline $\mathrm{Ex}^{(1)}(n;C)=\{S_n\}$.
\end{theorem}

 For large $n$, the Tur\'an numbers for longer (than three) loose 3-uniform paths were found in \cite{Kostochka}.
The case of length three has been omitted in \cite{Kostochka}, probably because the authors thought it had been taken care of in \cite{FJS}, where $k$-uniform loose paths were considered, $k\ge4$.
 However,  the method used in  \cite{FJS}
 did not quite work for 3-graphs. In \cite{JPR}  we fixed this omission. Given two $3$-graphs $F_1$
and $F_2$, by $F_1\cup F_2$  denote a vertex-disjoint union of $F_1$ and $F_2$. If $F_1=F_2=F$ we will sometimes write $2F$ instead of $F\cup F$.

\begin{theorem}[\cite{JPR}]\label{ex1}
 $$\ex^{(1)}(n;P)=\left\{ \begin{array}{ll}
 \binom n3 & \textrm{ and $\quad Ex^{(1)}(n;P)=\{K_n\}\qquad\quad$\;\; for $n\le6,$ }\\
20 & \textrm{ and $\quad Ex^{(1)}(n;P)=\{K_6\cup K_1\}\quad\,\,$ for $n=7,$ }\\
\binom{n-1}{2} & \textrm{ and $\quad Ex^{(1)}(n;P)=\{S_n\}\qquad\quad$\;\;\; for $n\ge 8$.}
\end{array} \right.
$$
\end{theorem}

\bigskip
Interestingly, although  the ordinary Tur\'an numbers for the 2-matching $M$ and the 3-path $P$ are equal for $n\ge8$, their higher order counterparts differ significantly, being, respectively,
 of linear and quadratic order in $n$.
In \cite{JPRr} we have  completely determined the second order Tur\'an number  $\ex^{(2)}(n;P)$, together with the corresponding
2-extremal 3-graphs. \emph{A comet} $\co(n)$ is an $n$-vertex 3-graph  consisting of the complete 3-graph
$K_4$ and the full star $S_{n-3}$, sharing exactly one vertex which is the center of the
star (see Fig. \ref{FigR1}). This vertex is called  \emph{the center} of the comet, while the set of
the remaining three vertices of the $K_4$ is called its \emph{ head}.

\bigskip
\begin{figure}[!ht]
\centering
\includegraphics [width=7cm]{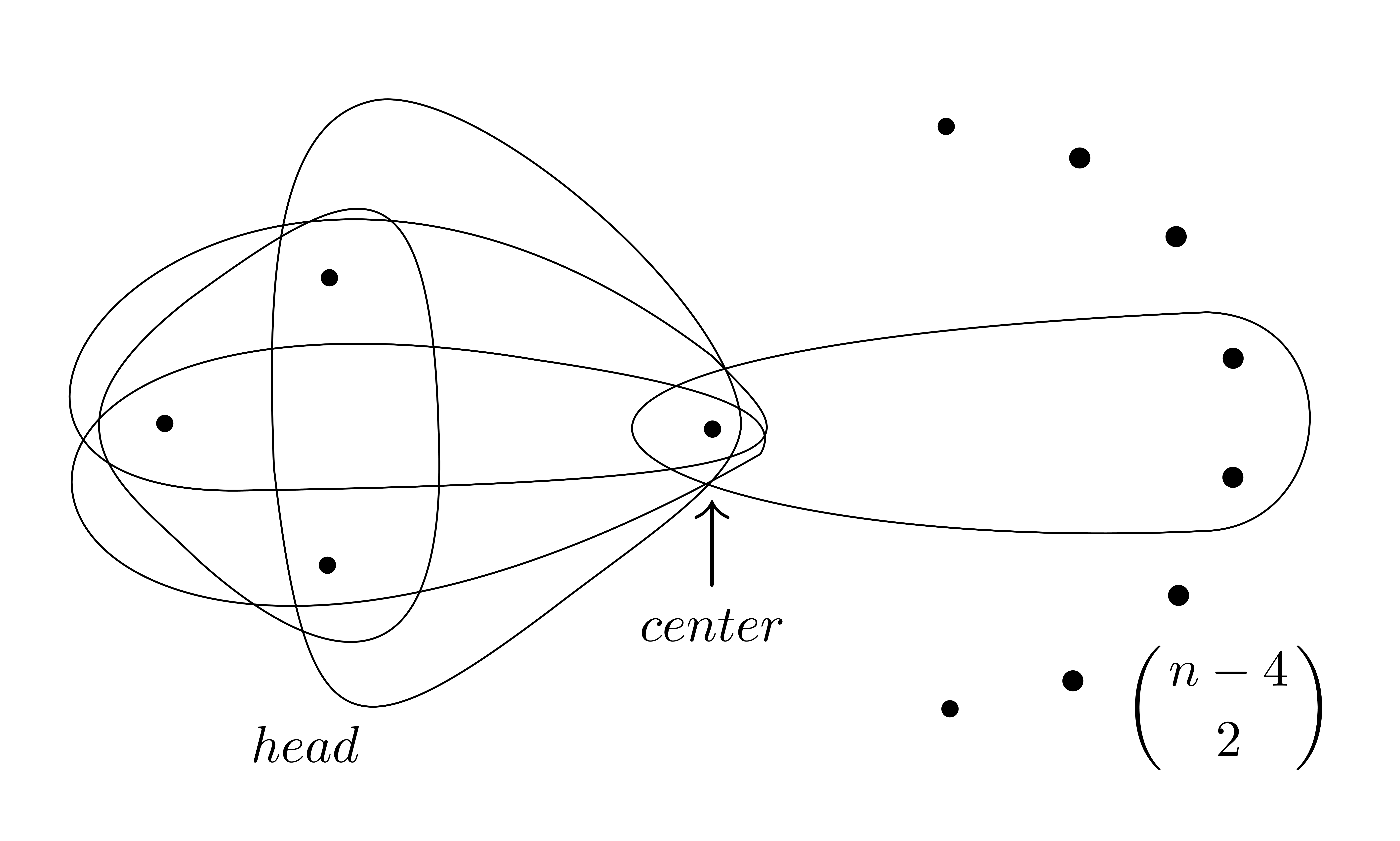}
\caption{The comet $\co(n)$} \label{FigR1}
\end{figure}

\begin{theorem}[\cite{JPRr}]\label{ex2}
    $$\ex^{(2)}(n;P)=\left\{ \begin{array}{ll}
    15 &\textrm{and} \quad  \Ex^{(2)}(n;P)=\{S_7\}\hskip 3.5cm \textrm{for $n=7$},\\
    20 + \binom {n-6}3 & \textrm{and}  \quad \Ex^{(2)}(n;P)=\{K_6 \cup K_{n-6}\}\hskip 1.05cm \textrm{for $8 \le n\le 12$},\\
    40 & \textrm{and}  \quad \Ex^{(2)}(n;P)=\{2K_6\cup K_1,\co(13)\} \hskip 0.6cm\textrm{for }n=13,\\
    4 + \binom{n-4}{2}&\textrm{and}\quad\Ex^{(2)}(n;P)=\{\co(n)\}\hskip 2.65cm\textrm{for }n\ge 14.
    \end{array} \right.
    $$
\end{theorem}

\noindent Note that for $n\le 6$ the second order number is not defined, since each 3-graph is a sub-3-graph
of~$K_n$. The main message behind  the above result is that for $n\ge8$ it provides an upper bound on the number of edges in  an $n$-vertex $P$-free 3-graph which is not a star.

Also in \cite{JPRr}, we managed to calculate the third order Tur\'an number for $P$ and $n=12$.
\begin{theorem}[\cite{JPRr}]\label{ex3}
    $$\ex^{(3)}(12;P)=32 \quad\text{and} \quad \Ex^{(3)}(12;P)=\{\co(12)\}.$$
\end{theorem}

The main Tur\'an-type result of this paper provides a complete formula for the third order Tur\'an number for $P$.

\begin{theorem}\label{ex3_cale}
	$$
	\mathrm{ex}^{(3)}(n;P)=
	\left\{\begin{array}{ll}
	3n-8 &\textrm{and} \quad \Ex^{(3)}(n;P)=\{G_1(n),G_2(n)\}\hskip 0.95cm\textrm{for }7\le n\le 10,\\
	25 &\textrm{and} \quad \Ex^{(3)}(n;P)=\{G_1(n),G_2(n),\co(n)\}\hskip 0.5cm\textrm{for } n= 11,\\
	32 &\textrm{and} \quad \Ex^{(3)}(n;P)=\{\co(n)\}\hskip 3cm\textrm{for }n=12,\\
	20 + \binom{n-7}{2}  &\textrm{and} \quad \Ex^{(3)}(n;P)=\{K_6\cup S_{n-6}\}\,\hskip 1.2cm\textrm{for }13\le n\le 14,\\
	4 + \binom{n-5}{2} &\textrm{and} \quad \Ex^{(3)}(n;P)=\{K_4\cup S_{n-4}\}\,\hskip 2.2cm\textrm{for }n\ge 15.
	\end{array}\right.
	$$
\end{theorem}
\noindent In particular, for $n\ge14$, Theorem \ref{ex3_cale} gives an upper bound  on the number of edges in  an $n$-vertex $P$-free 3-graph which is neither a star nor is contained in the comet $\co(n)$.

Surprisingly, as an immediate consequence we obtain an exact formula for the 4th Tur\'an number for $P$, at least for $n\ge15$. Indeed, consider the 3-graph $\ro(n)$, called \emph{rocket}, obtained from the star $S_{n-4}$ with center $x$ by adding to it 4 more vertices, say, $a,b,c,d$, and three edges: $\{x,a,b\},  \{a,b,c\}$, $\{a,b,d\}$. Clearly, $|\ro(n)|=|K_4~\cup~S_{n-4}|~-~1$, $\ro(n)\not\subset S_n$,  $\ro(n)\not\subset \co(n)$, and   $\ro(n)\not\subset K_4\cup S_{n-4}$. Hence, $\mathrm{ex}^{(4)}(n;P) \ge\mathrm{ex}^{(3)}(n;P)~-~1$, but, in view of inequality (\ref{decrease}), the two numbers  cannot be equal. 

In a similar fashion, by choosing an appriopriate 4-extremal 3-graph, one can show that $\mathrm{ex}^{(4)}(7;P)=	\mathrm{ex}^{(3)}(7;P)-1$ and 	$\mathrm{ex}^{(4)}(14;P)=	\mathrm{ex}^{(3)}(14;P)-1$.
With some additional effort we were also  able to calculate  the remaining values of 	$\mathrm{ex}^{(4)}(n;P)$ and  determine the  families $\Ex^{(4)}(n;P)$ of 4-extremal  3-graphs.
As these numbers are, however, of no use for calculating the respective Ramsey numbers, we state  Theorem \ref{ex4} without proof.
Let $K_5^{+2}$ be the 3-graph obtained from $K_5$ by fixing two of its vertices, say $a,b$, and adding two more vertices $c,d$ and two edges $\{a,b,c\}$ and $\{a,b,d\}$.

\begin{theorem}\label{ex4}
	$$
	\mathrm{ex}^{(4)}(n;P)=
	\left\{\begin{array}{ll}
	12 &\textrm{and} \quad \Ex^{(4)}(n;P)=\{G_3(n),K_5^{+2}\}\hskip 2.6cm\textrm{for }n=7,\\
	2n-2 &\textrm{and} \quad \Ex^{(4)}(n;P)=\{G_3(n)\}\hskip 2,82cm\textrm{for }8\le n\le 9,\\
	20 &\textrm{and} \quad \Ex^{(4)}(n;P)=\{K_5\cup K_5\}\hskip 2,95cm\textrm{for }n=10,\\
	20 &\textrm{and} \quad \Ex^{(4)}(n;P)=\{G_3(n)\}\hskip 3.4cm\textrm{for }n=11,\\
	28 &\textrm{and} \quad \Ex^{(4)}(n;P)=\{G_1(n),G_2(n)\}\hskip 2.15 cm\textrm{for } n= 12,\\
	33 &\textrm{and} \quad \Ex^{(4)}(n;P)=\{K_6\cup G_1(n),K_6\cup G_2(n)\}\hskip 0.2cm\textrm{for }n=13,\\
	40 &\textrm{and} \quad \Ex^{(4)}(n;P)=\{2K_6\cup 2K_1,K_4\cup S_{10})\}\hskip 0.65cm\textrm{for }n=14,\\
	48  &\textrm{and} \quad \Ex^{(4)}(n;P)=\{\ro(n),K_6\cup S_{9}\}\,\hskip 1.74cm\textrm{for }n=15,\\
	3 + \binom{n-5}{2} &\textrm{and} \quad \Ex^{(4)}(n;P)=\{\ro(n)\}\,\hskip 3.35cm\textrm{for }n\ge 16.
	\end{array}\right.
	$$
\end{theorem}

To determine Tur\'an numbers, it is sometimes useful to rely on Theorem \ref{ntif} and divide all 3-graphs into those which contain $M$ and those which do not.
To this end, it is convenient to define conditional Tur\'an numbers (see \cite{JPR, JPRr}).
     For a family of $3$-graphs $\mathcal F$, an $\mathcal F$-free $3$-graph $ G$, and an integer $n\ge |V(G)|$, the  \textit{conditional Tur\'an number} is defined as
     \begin{eqnarray*}
        \ex(n;\mathcal F|G)=\max\{|E(H)|:|V(H)|=n,\;
      \mbox{$H$ is
    $\mathcal F$-free, and } H\supseteq G \}
     \end{eqnarray*}
     Every $n$-vertex $\mathcal F$-free $3$-graph with $\ex(n;\mathcal F|G)$ edges and such
     that $H\supseteq G$  is called \emph{$G$-extremal for $\mathcal F$}.
     We denote by $\Ex(n;\mathcal F| G)$ the family of all  $n$-vertex $3$-graphs which are $ G$-extremal for $\mathcal F$.
 (If $\mathcal F=\{F\}$, we simply write $F$ instead of $\{F\}$.)

To illustrate the above mentioned technique, observe that for $n\ge7$
$$\mathrm{ex}^{(2)}(n;P)=\max\{ \ex(n; P|M), \mathrm{ex}^{(2)}(n;M)\}\overset{Thm\ref{ntif}
}{=}\max\{ \ex(n; P|M), 3n-8\}=\ex(n; P|M),$$
the last equality holding for  sufficiently large $n$ (see \cite{JPRr} for details).

In the  proof of Theorem \ref{ex3_cale} we will use the following three lemmas, all proved in \cite{JPRr}. For the first one we need one more piece of notation.
 If, in the above definition,  we restrict ourselves to connected $3$-graphs only (connected in the weakest, obvious sense) then the corresponding conditional Tur\'an number and the extremal family are denoted by
$ \ex_{conn}(n;\mathcal F|G)$ and  $\Ex_{conn}(n;\mathcal F| G)$, respectively.

\begin{lemma}[\cite{JPRr}]\label{spojny} For $n\ge7$,
    $$ \ex_{conn}(n; P|C)=3n-8\mbox{ and }\Ex_{conn}(n;P| C)=\{G_1(n),G_2(n)\}.$$
\end{lemma}

Lemma  \ref{spojny} as stated in \cite{JPRr} does not provide family $\Ex_{conn}(n;P| C)$. However, it is clear form its proof that the $C$-extremal 3-graphs are the same as in Theorem \ref{ntif}.

\begin{lemma}[\cite{JPRr}]\label{PCM}
    $$
    \mathrm{ex}(n;\{P,C\}|M)=
    \left\{\begin{array}{ll}
    2n-4 &\qquad\qquad\qquad\qquad\qquad\qquad\qquad\;\;\;\;\textrm{for } 6\le n\le 9,\\
    20 &\qquad\qquad\qquad\qquad\qquad\qquad\qquad\qquad\;\;\textrm{for }n=10,\\
    4 + \binom{n-4}{2} &\textrm{and} \quad \Ex (n;\{P,C\}|M)=\{\co(n)\}\quad\quad\textrm{for } n\ge 11.
    \end{array}\right.
    $$
\end{lemma}

\begin{lemma}[\cite{JPRr}]\label{pcppm}
    For $n\ge 6$
    $$\ex(n;\{P,C,P_2\cup K_3\}|M)=2n-4,$$
where $P_2$ is a pair of edges sharing one vertex.
\end{lemma}

\section{Proof of Theorem \ref{main}}\label{proofRam}

As mentioned in the Introduction, the inequality $R(P;r)\ge r+6$, $r\ge1$, has been already observed
in \cite{J}. We are going to show that $R(P;r)\le r+6$ for $8\le r\le 9$. Along the way, we  need to strengthen the results for  $3\le r\le 7$ as follows.
Let $K_n-e$ denote the 3-graph with $n$ vertices and $\binom n3-1$ edges, while $K_n-2e$ denote each of the three possible (up to isomorphism) 3-graphs with $n$ vertices and $\binom n3-2$ edges.
Write $H\to(P;r)$ if every $r$-coloring of the edges of $H$ yields a monochromatic copy of $P$.

\begin{lemma}\label{strength}
For every $n=9,\dots,13$, $K_n-2e\to(P;n-6)$.
\end{lemma}

\proof
 A coloring which does not yield a monochromatic copy of $P$ is referred to as \emph{proper}. Below, for each $n$ we assume that there is a proper coloring of $K_{n}-2e$ and arrive at a contradiction.

\medskip

\noindent{\bf$\mathbf{n=9}$:} For every 3-coloring of $K_9-2e$ there is a color with at least 28 edges and thus, if it is proper, then, by Theorem \ref{ex1}, that color must form  a full star. After removing the center of that star, we obtain a proper 2-coloring of $K_8-2e$, which, again by Theorem \ref{ex1}, contains a monochromatic copy of $P$, a contradiction.

\medskip

\noindent{\bf$\mathbf{n=10}$:} For every 4-coloring of $K_{10}-2e$ there is a color with at least 30 edges and thus, if it is proper, then, by Theorems \ref{ex1} and  \ref{ex2}, that color must form  a  star.  After removing the center of that star we are back to the $n=9$ case.

\medskip

\noindent{\bf$\mathbf{n=11}$:} For every 5-coloring of $K_{11}-2e$ there is a color with at least 33 edges and thus, if it is proper, then, again by Theorems \ref{ex1} and  \ref{ex2}, that color must form  a  star.
 After removing the center of that star we are back to the $n=10$ case.

\medskip

We leave the most difficult case of $n=12$ to the end of the proof.

\medskip

\noindent{\bf$\mathbf{n=13}$:} For every 7-coloring of $K_{13}-2e$ there is a color with at least 41 edges and thus, if it is proper, then, one more time by Theorems \ref{ex1} and  \ref{ex2}, that color must form  a  star.
 After removing the center of that star we jump to the $n=12$ case.

\medskip

\noindent{\bf$\mathbf{n=12}$:} Consider a 6-coloring of $K_{12}-2e$. If a color  forms a star,  then, after removing its center, we  obtain a 5-coloring of  $K_{11}-2e$ (or  $K_{11}-e$ or $K_{11}$) which, as we have already proved, contains  o monochromatic copy of $P$. Thus, from now on we assume no color class forms a star.
However, there is a color with at least 37 edges and thus, if it is proper, then, this time by Theorems \ref{ex1}-\ref{ex3}, that color must be a sub-3-graph of $K_6\cup K_6$. After removing that copy of $K_6\cup K_6$, we are looking at a proper 5-coloring of a complete,  6 by 6  bipartite 3-graph $H$, with possibly up to  two edges missing. As $|H|\ge 220-40-2=178$, the average number of edges per color is at least 35.6. On the other hand no color may have been applied to more than 36 edges. The reason is that, as explained above, such a color class would need to be a sub-3-graph of a copy of $K_6\cup K_6$, but it is easy to check that every copy of $K_6\cup K_6$ shares at most 36 edges with $H$. This implies that among the five color classes at least three have size 36. But it was shown already in \cite{JPRr} (proof of Theorem 1, case $r=6$) that the coexistence of three disjoint sub-3-graphs in $H$, each having 36 edges and  contained in a copy of $K_6\cup K_6$, is impossible. \qed

\bigskip

\noindent{\bf{\textit{Proof of Theorem \ref{main}}}:}
In the case $r=8$ we are going to prove a little stronger result to be used in the case $r=9$.

\medskip

\noindent{\bf$\mathbf{r=8}$:} For every 8-coloring of  $K_{14}-e$ there is a color with at least 46 edges,  and thus, if it is proper, then,  by Theorems \ref{ex1}, \ref{ex2}, and \ref{ex3_cale},  that color must either form  a  star or be a sub-3-graph of the comet $\co(14)$. In either case, we remove the center of the structure in question, a star or a  comet, and in addition, if it is the comet, the edge spanned by its head. We get  at a 7-coloring of $K_{13}-e$ or $K_{13}-2e$ which, by Lemma \ref{strength}, yields a monochromatic copy of $P$, a contradiction.
 It follows that $K_{14}-e\to (P;8)$, and so,  $K_{14}\to (P;8)$ too. Hence, we have proved Theorem \ref{main} for $r=8$.

 \medskip

\noindent{\bf$\mathbf{r=9}$:}
For every 9-coloring of  $K_{15}$ there is a color with at least 51 edges, and thus, if it is proper, then, again by Theorems \ref{ex1}, \ref{ex2}, and  \ref{ex3_cale}, that color   must form a sub-3-graph
      of $S_{15}$  or $\co(15)$. Similarly to the case $r=8$, we remove a vertex and possibly an edge, to obtain an 8-coloring of $K_{14}$ or $K_{14}-e$, which, by the case $r=8$ yields a monochromatic copy of $P$, a contradiction.
\qed

\section{Proof of Theorems \ref{ex3_cale}}

  For $n=12$, Theorem \ref{ex3_cale} has been already proved in \cite{JPRr} (cf. Theorem \ref{ex3} in Introduction). Therefore, it seems natural  to divide the  proof into two ranges of $n$: smaller than 12 and larger than 12. The general set-up is, however, the same for both. We first check that all candidates for being 3-extremal 3-graphs do  qualify, that is, are $P$-free, are not contained in any of the 1-extremal or 2-extremal 3-graphs with the same number of vertices, and have the  number of edges given by the formula to be proved. Then, we  consider an arbitrary $n$-vertex, qualifying 3-graph $H$ and show that unless it is one of the candidate 3-extremal  3-graphs itself, its number of edges is strictly smaller than theirs.

For the latter task, we distinguish  two cases: when $H$ is connected and disconnected. The entire proof is inductive, in the sense that here and there we  apply the very Theorem~\ref{ex3_cale}
for smaller instances of $n$, once they have been confirmed.

\subsection{$\mathbf{7\le n \le 11}$} First note that by Theorems \ref{ex1} and \ref{ex2}, for $7 \le n \le 11$
$$
\Ex^{(1)}(n;P) \cup \Ex^{(2)}(n;P) =\{S_{n}, K_6\cup K_{n-6}\}.
$$
Moreover, for $i\in\{1,2\}$, $G_i(n)\nsubseteq S_{n}$ and $G_i(n)\nsubseteq K_6\cup K_{n-6}$. Consequently, since $G_i(n)$ is $P$-free,
$$
\ex^{(3)}(n;P)\ge |G_i(n)|=3n-8.
$$
We are going to show that, in fact, the 3-graphs $G_i(n)$, $i\in \{1,2\}$ are the only 3-extremal 3-graphs for $n\le 10$, whereas, for $n=11$, in addition, the comet $\co(11)\in \Ex^{(3)}(11;P)$.

For $7\le n \le 11$, let $H$ be an $n$-vertex $P$-free 3-graph, satisfying
\begin{equation}\label{nsub}
H\nsubseteq S_{n}\qquad\mbox{and}\qquad  H\nsubseteq K_6\cup K_{n-6}.
\end{equation}
We first assume that $H$ is \emph{connected}. If $H\supset C$ or $H\not\supset M$ then, by, respectively, Theorem \ref{ntif} or Lemma \ref{spojny},
$$
|H|\le 3n-8,
$$
with the equality for $H=G_i(n)$, $i\in \{1,2\}$ only. Otherwise,
$$|H|\le \ex(n;\{P,C\}|M)\le 3n-8,$$
where the second inequality holds by Lemma \ref{PCM}, and it becomes equality only if $n=11$ and $H=\co(11)$.

Next, we show  that if a $P$-free 3-graph $H$ satisfying (\ref{nsub}) is \emph{disconnected}, then $|H|< 3n-8$.  Let $m=m(H)$ be the number of vertices in the smallest componet of $H$. Since $n\le11$, we have $m\le5$. Moreover, $m\neq 2$, since no component of a 3-graph may have two vertices. Thus,  $m\in\{1,3,4,5\}$.
Note also that, as a consequence of the second part of  (\ref{nsub}),  no union of components of $H$ may have  6 vertices together. Consequently, $m\neq n-6$. We now break the proof into several cases.

If $n=7$,  we must have $m=3$ and thus,
$$|H|\le1+4<3n-8=13.$$
 For $n\ge8$, if $m=1$, that is, if there is an isolated vertex $v$ in $H$, then
$H-v$ still satisfies (\ref{nsub}) with $n-1$ instead of $n$, and, by induction,
$$|H|\le \ex^{(3)}(n-1;P)\le 3(n-1)-8<3n-8.$$
 If $m=3$, then, for $n=8,10,11$,we apply the bound
$$|H|\le 1+ \ex^{(1)}(n-3;P)< 3n-8,$$
where the  last inequality follows by Theorem \ref{ex1}. If $m=4$, we have, similarly,
$$|H|\le 4+ \ex^{(1)}(n-4;P)< 3n-8,$$
for $n=8,9,11$. Finally, if $m=5$ (and $n=10$), $|H|\le 2\times\binom 53<22$.

\subsection{$\mathbf{n\ge 13}$}
By Theorems \ref{ex1} and \ref{ex2},
$$
\Ex^{(1)}(13;P) \cup \Ex^{(2)}(13;P) =\{S_{13}, \co(13),K_6\cup K_6\cup K_1\},
$$
while for $n\ge 14$,
$$
\Ex^{(1)}(n;P) \cup \Ex^{(2)}(n;P) =\{S_n, \co(n)\},
$$
Therefore, to determine $\ex^{(3)}(n;P)$ for $n\ge 13$ we have to find the largest number of edges in an $n$-vertex  $P$-free 3-graph $H$ such that $H\nsubseteq S_{n}$, $H \nsubseteq \co(n)$ and for $n=13$, in addition, $H\nsubseteq K_6\cup K_6\cup K_1$.
The 3-graphs
$$H_{n}:=K_6\cup S_{n-6}\quad \mbox{for}\quad n\in\{13,14\}\qquad\mbox{ and}\qquad H_n:=K_4\cup S_{n-4}\quad\mbox{ for}\quad n\ge 15$$ satisfy all the above conditions.
Hence, for $n\ge13$,
$$
\ex^{(3)}(n;P)\ge |H_n|.
$$
We are going to show that also the opposite inequality holds, as well as, that the equality holds  for $H_n$  only.

\bigskip
 To this end, let $H\neq H_n$ be an $n$-vertex
$P$-free 3-graph such that

\begin{equation}\label{nsub1}
H\nsubseteq S_{n},\qquad\mbox{and}\qquad  H\nsubseteq \co(n),
\end{equation}
 and for $n=13$, in addition, $H\nsubseteq K_6\cup K_6\cup K_1$. We will show that $|H|< |H_n|$.

Assume first that $H$ is \emph{ connected}.  If, in addition, $H\supset C$ or $H\not\supset M$, then,   by, respectively, Lemma \ref{spojny} or Theorem \ref{ntif},
$$
|H|\le  3n-8  < |H_n|.
$$
Otherwise, $H$ is a $\{P,C\}$-free, connected 3-graph containing $M$.
Since, by Lemma \ref{pcppm},
$$
\ex(\{n;P,C,P_2\cup K_3\}|M)=2n-4 < |H_n|,
$$
we may assume that $ P_2\cup K_3\subset H$.
Thus, the connected case will be completed once we have proved the following lemma.

\begin{lemma}\label{con}
If $H$ is a connected, $n$-vertex, $n\ge13$, $\{P,C\}$-free 3-graph such that  $H\supset P_2\cup K_3$,  and $H \nsubseteq \co(n)$, then $|H|<|H_n|$.
\end{lemma}
\noindent (Note that  for $n=13$, the  additional requirement  that $H\nsubseteq K_6\cup K_6\cup K_1$ is, due to connectedness, fulfilled automatically.)  We devote an entire  section to prove Lemma \ref{con}. Meanwhile, taking Lemma \ref{con} for granted, let us quickly complete the proof of Theorem~\ref{ex3_cale}.

Assume that $H$ is \emph{disconnected} and, as before, let $m=m(H)$ be the order of the smallest component of $H$,  $1\le m\le n-m$, $m\neq2$. Below, to bound $|H|$, we use the Tur\'an numbers for $P$ of the 1st, 2nd, and 3rd order and utilize, respectively, Theorems \ref{ex1}, \ref{ex2}, and \ref{ex3_cale} (per induction).

If $v$ is an isolated vertex ($\mathbf{m=1}$), then, similarly as for small $n$,  $H-v$ satisfies (\ref{nsub1}), because otherwise $H$ would not.
 Hence, for $n\ge 15,$
$$
|H|\le \ex^{(3)}(n-1;P)<|H_n|.
$$
For $n=14$, we cannot guarantee that $H-v\not\subset K_6\cup K_6\cup K_1$, so we use the second order Tur\'an number instead which still does  the job:
$$
|H|\le \ex^{(2)}(13;P)=40<41=|H_{14}|.
$$
To complete the case $m=1$ notice that for $n=13$, since $H\nsubseteq K_6\cup K_6\cup K_1$, we have $H-v\not\subset K_6\cup K_6$, and we are in position to use  induction again. Hence,
$$
|H|\le \ex^{(3)}(12;P)=32<35=|H_{13}|.
$$

 For $m\ge3$, let us express $H$ as a vertex disjoint union of two 3-graphs:
$$
H=H'\cup H'', \qquad |V(H')| = m, \quad |V(H'')|=n-m
$$
Then, clearly, both $H'$ and $H''$ are $P$-free, and thus
\begin{equation}\label{ml}
|H|\le \ex^{(1)}(m;P)+\ex^{(1)}(n-m;P).
\end{equation}

For $\mathbf{m=3}$, since  $H\nsubseteq \co(n)$, we have $H''\nsubseteq S_{n-3}$ and consequently
$$
|H|\le 1+ \ex^{(2)}(n-3;P)<|H_n|,
$$
where the last inequality is easily verified by hand.

For $\mathbf{m=4}$ and $n\in \{13,14\}$ by (\ref{ml}),
$$
|H|\le 4 + \ex^{(1)}(n-4;P)= 4+ \binom{n-5}{2}< 20 + \binom{n-7}{2}=|H_n|;
$$
and for $n\ge 15$, either  $H''\subseteq S_{n-4}$ and so $H\subseteq K_4\cup S_{n-4}=H_n$ (in which case we are done) or  $H''\nsubseteq S_{n-4}$ but then,
$$
|H|\le 4 + \ex^{(2)}(n-4;P)<4+ \binom{n-5}{2}=|H_n|.
$$

For $\mathbf{m=5}$ by (\ref{ml}),
$$
|H|\le  \ex^{(1)}(5;P)+ \ex^{(1)}(n-5;P)= 10+ \binom{n-6}{2}< |H_n|.
$$

For $\mathbf{m=6}$ and $n=13$, since $H\nsubseteq K_6\cup K_6\cup K_1$ we have $H''\not\subset K_6\cup K_1$ and so, $|H''|\le \ex^{(2)}(7;P)=|S_7|$ whereas for $n\ge 14$, we bound
 $|H''|\le \ex^{(1)}(n-6;P)=|S_{n-6}|$.  Hence, in both cases we have
$$
|H|\le  \ex^{(1)}(6;P)+ |S_{n-6}|= 20+ \binom{n-7}{2}\le |H_n|.$$
However, for $n\in\{13,14\}$, the first inequality must be strict (since $H\neq H_n$), while for $n\ge15$ the second inequality is strict.

For $\mathbf{m=7}$ we have $n\ge 14$  and, by (\ref{ml}), for $n=14$
$$
|H|\le  \ex^{(1)}(7;P)+ \ex^{(1)}(7;P)= 20+ 20<41= |H_{14}|,
$$
whereas, for $n\ge 15$
$$
|H|\le  \ex^{(1)}(7;P)+ \ex^{(1)}(n-7;P)= 20+ \binom{n-8}{2}<4+\binom{n-5}{2} = |H_n|.
$$

Finally, for $\mathbf{m\ge 8}$ we have $n\ge 16$ and, by (\ref{ml}),
$$
\begin{aligned}
|H|\le \ex^{(1)}(m;P)+ \ex^{(1)}(n-m;P)=\binom{m-1}{2}+ \binom{n-m-1}{2}\le\binom 72+\binom{n-9}2 \\
<\binom 32+\binom{n-5}2
<4+\binom{n-5}{2} = |H_n|.
\end{aligned}
$$

\bigskip

\subsection{Preparations for the proof of Lemma \ref{con}}\label{prepa}

Under the assumptions on $H$ in Lemma \ref{con}, let $Q$ be a copy of $P_2$ in $H$ such that there is at least one edge disjoint from $U=V(Q)$.
We know that $Q$ exists, because  $ P_2\cup K_3\subset H$. Let $V=V(H)$ and $W=V\setminus U$.
  Further, let $W_0$ be the set of vertices of degree zero in $H[W]$ and $W_1=W\setminus W_0$.
(see Fig. \ref{FigR6}).  Note that, by definition,  $H[W]=H[W_1]$ and $|W_1|\ge3$.

\bigskip
\begin{figure}[!ht]
\centering
\includegraphics [width=9cm]{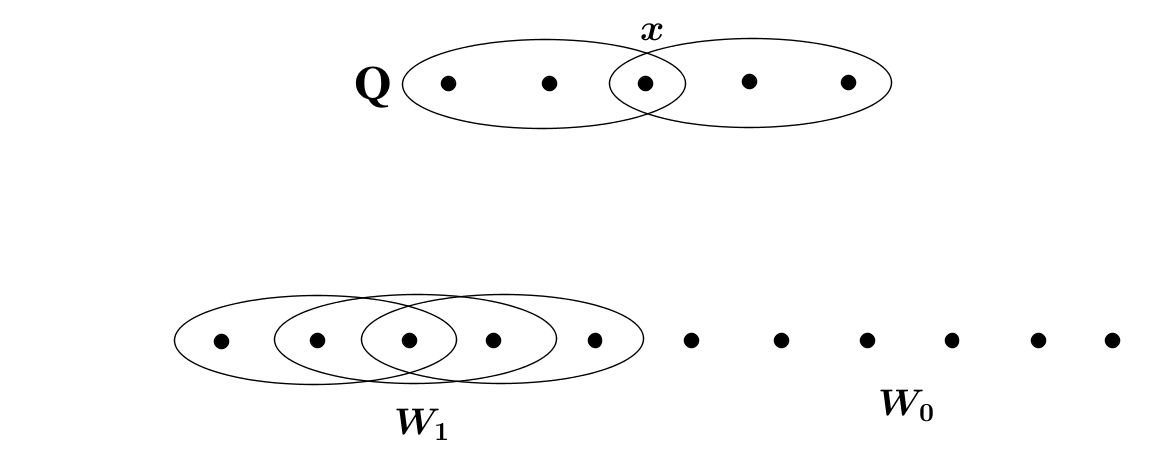}
\caption{Set-up for the proof of Lemm \ref{con}} \label{FigR6}
\end{figure}

\medskip

\noindent We also split the set of edges of $H$.
First, notice that, since $H$ is $P$-free, there is no edge with one vertex in each $U$, $W_0$, and $W_1$. Let for $i=0,1$, $H_i$ be the sub-3-graph of $H$ composed of the edges intersecting both $U$ and $W_i$. Then, clearly,
\begin{equation}\label{HHH}
    H=H[U]\cup H[W]\cup H_0\cup H_1,
\end{equation}
with all four parts edge-disjoint.

Let $x$ be the vertex of degree two in $Q$. If for some $h\in H_0\cup H_1$ we have $|h\cap U|=1$, then $h\cap U = \{x\}$, since otherwise $h$ together with $Q$ would form a copy of $P$ in $H$. We let
$$
F^0=\{h\in H_0\cup H_1: h\cap U=\{x\}\}.
$$

Also, the edges $h\in H_0\cup H_1$ with $|h\cap U|=2$ must be such that the pair $h\cap U$ is contained in an edge of $Q$,
since otherwise $h$ together with  $Q$ would form a copy of $C$ in $H$. For $k=1,2$, define
$$
F^k=\{h\in H_0\cup H_1: \quad |h\cap U\setminus\{x\}|=k\}.
$$
We have $H_0\cup H_1=F^0\cup F^1\cup F^2$ (see Fig.
\ref{FigR7}.) Further, for $i=0,1$ and $k=0,1,2$, we set
$$F_i^k=F^k\cap H_i.$$

\bigskip
\begin{figure}[!ht]
\centering
\includegraphics [width=7cm]{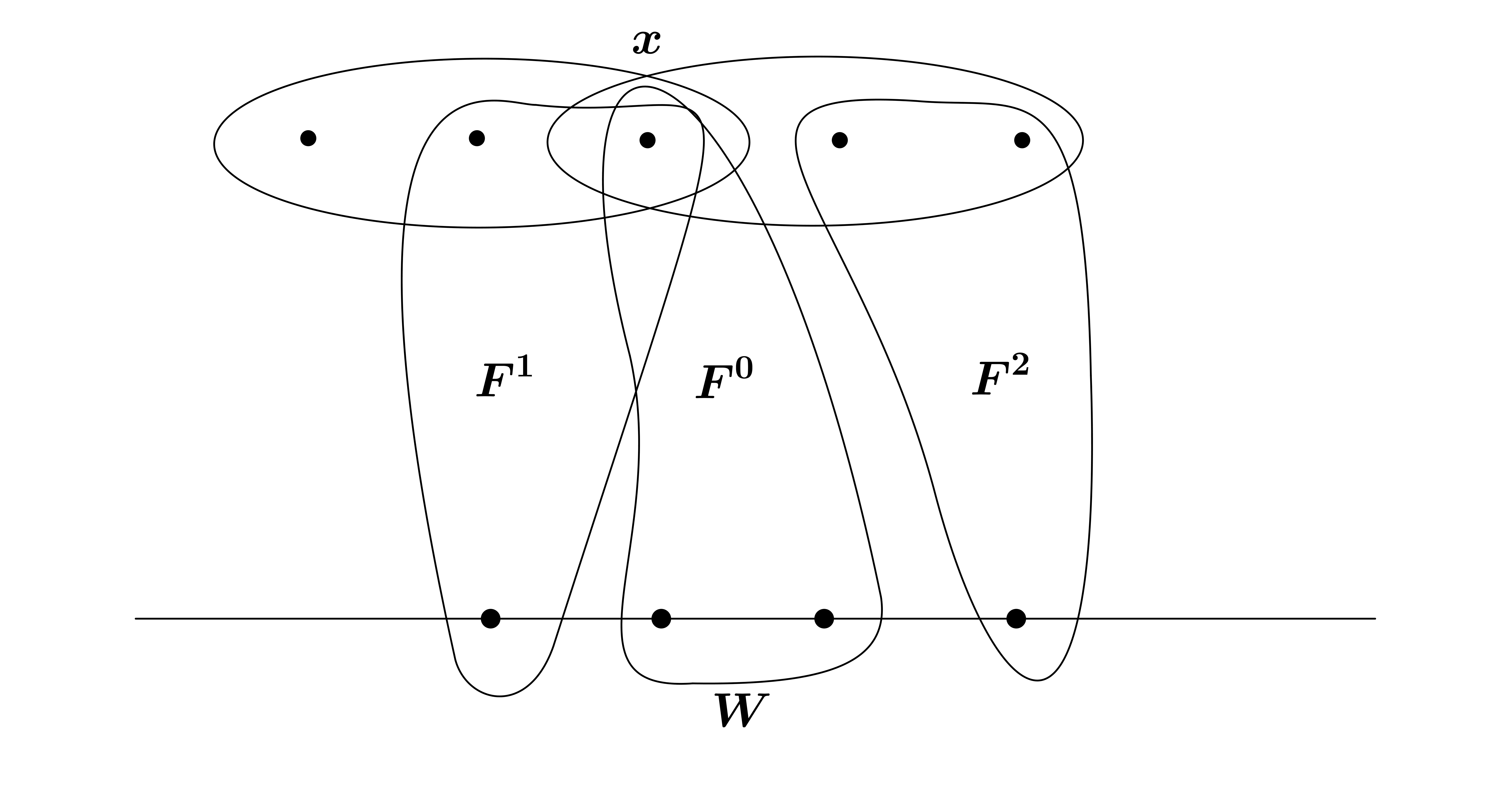}
\caption{Three types of edges in $H_0\cup H_1$} \label{FigR7}
\end{figure}

   \noindent  Note that, since $H$ is $P$-free, $F^1_1=\emptyset$ and thus,
\begin{equation}\label{r6}
H_1=F_1^0\cup F_1^2.
\end{equation}
Observe also that, because $H$ is connected, $H_1\neq\emptyset$.
 Consequently, since the presence of any edge of $H_1$ forbids at least 4 edges of $H[U]$,
 \begin{equation}\label{hu1}
 |H[U]|\le 6.
 \end{equation}
\medskip

In \cite{JPRr} the authors have proven the following bound on the number of edges in~$H_1$:

\begin{equation}\label{r4}
|H_1|\le 2|W_1|-3.
\end{equation}
We use  (\ref{r4}) to estimate $|H[W]| + |H_1|$.
\begin{fact}\label{o5}
Set $|W_1|=z$. We have
\begin{equation}\label{e5}
|H[W]| + |H_1|\le
\left\{\begin{array}{ll}
\binom{z}{3} + 2z -3 &\textrm{for } 3\le z\le 5,\\
\binom{z-1}{2} + 2z -3 &\textrm{for } 6\le z\le 7,\\
\frac{(z-1)^2}{2}+2 &\textrm{for } z\ge 8.
\end{array}\right.
\end{equation}
\end{fact}
\proof For $3\le z \le 5$, the above inequality is an immediate consequence of (\ref{r4}),  whereas for $6\le z \le 7$, we use  (\ref{r4}) and the bound $|H[W]|\le\binom{z-1}2$, stemming from Theorem~\ref{c3}. For $z\ge 8$ we have to consider two cases.
Suppose first that $H[W]\subseteq S_{z}$ with the center~$v\in W_1$. Since $H$ is $P$-free, every  edge $h\in F^2_1$  must have
    $h\cap W_1= \{v\}$. Hence, $
    |F^2_1|\le 2.
    $
   Moreover,  if $e\in F^0_1$, then the pair $e\cap W_1$ must be \emph{nonseparable} in $H[W_1]$, that is, every edge of $H[W_1]$ must contain both these vertices or none.  Since, as it can be easily proved, there are at most $\left\lfloor\frac{z-1}{2}\right\rfloor$ nonseparable pairs in $W_1$,
   $$
   |F^0_1|\le \left\lfloor\frac{z-1}{2}\right\rfloor.
   $$
  Consequently,
 \begin{equation}\label{eq7}
    |H_1|\le 2+ \left\lfloor\frac{z-1}{2}\right\rfloor
    \end{equation}
and, again using Theorem \ref{c3},

\begin{equation}\label{e8}
|H[W]|+|H_1|\le \binom{z-1}{2} + 2+ \left\lfloor\frac{z-1}{2}\right\rfloor  \le \frac{(z-1)^2}{2}+2.
\end{equation}
Otherwise, either $H[W]\not\supseteq M$, and then, by Theorem \ref{ntif},
\begin{equation}\label{e6}
|H[W]|\le 3z-8,
\end{equation}
or $H[W] \supseteq M$ and then, by Lemma \ref{PCM},
\begin{equation}\label{e7}
|H[W]|\le \mathrm{ex}(n;\{P,C\}|M)=
\left\{\begin{array}{ll}
2z-4 &\textrm{for } 8\le z\le 9,\\
20 &\textrm{for }z=10,\\
4 + \binom{z-4}{2} &\textrm{for } z\ge 11.
\end{array}\right.
\end{equation}
Consequently, by (\ref{r4}),
\begin{equation}\label{e9}
|H[W]|+|H_1|\le
\left\{\begin{array}{ll}
4z-7 &\textrm{for } 8\le z\le 9,\\
37 &\textrm{for } z=10,\\
\binom{z-4}{2} + 2z +1&\textrm{for } z\ge 11,
\end{array}\right.
\end{equation}
and it is easy to check that for $z\ge 8$, the R-H-S of (\ref{e9}) is smaller than
$ \frac{(z-1)^2}{2}+2.
$ \qed

\subsection{Proof of Lemma \ref{con}}
We adopt the notation from  the previous subsection. In particular, recall that $z=|W_1|$. Additionally, we set $s=|W_0|$.
Our plan is to first give the proof in three `smallest'  cases:  $s=0$,  $z=3$, and $n\in\{13,14,15\}$.

\noindent{\bf$\mathbf{s=0}$ ($W_0=\emptyset$).}  Then $H_0=\emptyset$,  $z=n-5\ge 8$, and, by (\ref{HHH}), (\ref{hu1}), and  (\ref{e5}),

$$
|H| = |H[U]|+|H[W]| + |H_1|\le 6 + \frac{(z-1)^2}{2}+2.
$$
 This implies that for $n=13$, $|H|\le 32<35$, for $n=14$, $|H|\le 40<41$, while for $n\ge 15$, it can be easily checked that
 $$|H|\le  6 + \frac{(z-1)^2}{2}+2< 4+\binom{z}{2} = |H_n|.$$
 
 Therefore, from now on we will be assuming that
 $W_0\neq \emptyset,$ or $s\ge1$.
The proofs of the other two special cases, $z=3$ and $n\in\{13,14,15\}$, are both split into two subcases with respect to $F^2_1$.
 We begin with bounding the number of edges in $H[U\cup W_0]$ when $F^2_1\neq \emptyset$.
 
 \begin{fact}\label{UW0} If  $F^2_1\neq \emptyset$, then
 \begin{equation}\label{e4}
|H[U\cup W_0]|\le
\left\{\begin{array}{ll}
8 &\textrm{for } s=1,\\
3s+7 &\textrm{for } 2\le s\le 4,\\
\binom{s+2}{2}+1 &\textrm{for } s\ge 5.
\end{array}\right.
\end{equation}
 \end{fact}
\proof 
Let $u$ and $v$ be some two vertices of $U$ belonging to the same edge of $F^2_1$.
If $H[U\cup W_0]\subseteq S_{s+5}$ (with the center in $x$), then, since $H$ is $P$-free, the only edge of $H[U\cup W_0]$ containing  $u$ or $v$ is $\{x,u,v\}$. Hence
\begin{equation}\label{e1}
|H[U\cup W_0]|\le \binom{s+2}{2}+1.
\end{equation}
Otherwise, either $ H[U\cup W_0]\not \supseteq M$ and, assuming that $s\ge2$, and thus $|U\cup W_0|\ge7$, by Theorem \ref{ntif},
\begin{equation}\label{e2}
|H[U\cup W_0]|\le \mathrm{ex}^{(2)}(s+5;M)=3(s+5)-8=3s+7,
\end{equation}
or $ H[U\cup W_0] \supseteq M$ and, by Lemma \ref{PCM}, this time including $s=1$,
\begin{equation}\label{e3}
|H[U\cup W_0]|\le \mathrm{ex}(s+5;\{P,C\}|M)=
\left\{\begin{array}{ll}
2s+6 &\textrm{for } 1\le s\le 4,\\
20 &\textrm{for }s=5,\\
4 + \binom{s+1}{2} &\textrm{for } s\ge 6.
\end{array}\right.
\end{equation}
For $s=1$ we argue  as follows. Since  $H$ is $P$-free, every edge of $H[U\cup W_0]$ must contain either both of $u$ and $v$ or none.
 There are only 8 such edges and so, $|H[U\cup W_0]|\le 8$. In summary, (\ref{e4}) follows by (\ref{e1}), (\ref{e2}),  (\ref{e3}), and the above bound  for $s=1$. \qed

\bigskip

\noindent{\bf$\mathbf{z=3}$,\;$\mathbf{F^2_1\neq \emptyset}$.} 
We  combine bounds (\ref{e5}) of Fact \ref{o5} and (\ref{e4})  to estimate $|H|$.
Since $s=n-5-z\ge 13-8=5$,
$$
\begin{aligned}
|H|=|H[U\cup W_0]| +|H_1|+|H[W]| &\le \binom{s+2}{2}+1+\binom{3}{3} + 2\cdot 3 -3=\\
\binom{n-6}{2}+5&< |H_n|.
\end{aligned}
$$

\bigskip

\noindent{\bf$\mathbf{z=3}$,\;$\mathbf{F^2_1= \emptyset}$.} Then, by (\ref{r6}), $|H_1|\le |F_1^0|\le3$.
Since $H\nsubseteq \co(n)$ we have $H[U\cup W_0]\nsubseteq S_{s+5}$ and consequently, by Theorem \ref{ex2},
$$
H[U\cup W_0]\le \ex^{(2)}(s+5;P)=\left\{\begin{array}{ll}
20 + \binom{s-1}{3} &\textrm{for } \quad10\le s+5\le 12,\\
4+\binom{s+1}{2} &\textrm{for } \quad s+5\ge 13
\end{array}\right.
$$
Hence, for $13\le n \le 15$ ($10\le s+5\le 12$)
$$
|H|=|H[U\cup W_0]| +|H_1|+|H[W]| \le 20 + \binom{s-1}{3}+3+1=24+\binom{n-9}{3}<|H_n|,
$$
while for $n\ge16$ ($s+5\ge13$)
$$
\begin{aligned}
|H|=|H[U\cup W_0]| +|H_1|+|H[W]| &\le 4+\binom{s+1}{2}+3+1=\\
8+ \binom{n-7}{2}&< 4+\binom{n-5}{2}=|H_n|.
\end{aligned}
$$
Consequently for the rest of the proof we will be assuming that $z\ge 4$ (and $s\ge1$).
\medskip

\bigskip

\noindent{\bf$\mathbf{n\in\{13,14,15\}}$,\;$\mathbf{F^2_1\neq \emptyset}$.}
We again  combine bounds (\ref{e5}) of Fact \ref{o5} and (\ref{e4})  to estimate $|H|$.
For $n=13=5+s+z$, where $4\le z\le 7$, the worst case is when  $z=7$ and $s=1$, in which we get
$$|H|\le 34<35.$$
For $n=14=5+s+z$, where $4\le z\le 8$, the worst case is when  $z=7$ and $s=2$, and so
$$|H|\le 39<41.$$
For $n=15=5+s+z$, where $4\le z\le 9$,
 $$|H| \le 42<49.$$

\noindent{\bf$\mathbf{n\in\{13,14,15\}}$,\;$\mathbf{F^2_1=\emptyset}$.}
For $z=4$ by an easy inspection one can show that $|H_1| + |H[W]| \le 1 + 2=3$. Also, since $H$ is $C$-free, by Theorem~\ref{c3}, $|H[U\cup W_0]|\le\binom{s+4}2$.
Therefore, for $n=13$, $|H|\le 28+3<35$, for $n=14$, $|H|\le 36+3<41$, and for $n=15$, $|H|\le 45+3<49$.

\medskip

Finally, for $z\ge5$, we may apply  Theorem~\ref{c3} to $H[W_1\cup\{x\}]$, obtaining the bound
$$
|H_1| + |H[W]| = |H[W_1\cup\{x\}]|\le \binom{z}{2}.
$$
In summary,
$$
|H|=|H[U\cup W_0]| +|H_1|+|H[W]| \le \binom{s+4}{2} + \binom{z}{2},
$$
and consequently, by choosing optimal pairs $(z,s)$,
 for $n=13=5+s+z$, where $5\le z\le 7$, we get
 $$|H|\le 31<35,$$
 for $n=14=5+s+z$, where $5\le z\le 8$, we get
 $$|H|\le 38<41,$$
 whereas for $n=15=5+s+z$, where $5\le z\le 9$,
 $$|H| \le 46<49.$$

 Thus, we are done with the proof of Lemma \ref{con} in all three cases:  $s=0$,  $z=3$, and $n\in\{13,14,15\}$.  In fact, recalling our argument from Section 4.2, we have actually proved Theorem \ref{ex3_cale} for all $n\le15$.
 To complete the proof of Lemma \ref{con} for  the remaining values of $n$, we need only to prove Fact \ref{final} below. The proof is by induction on $n$, and we include the case $n=15$ there to serve as the inductive base.

 Note that compared to Lemma \ref{con}, we now relax the assumption of connectivity, replacing it with that of $H_1\neq\emptyset$. Also, although we have already proved Lemma \ref{con} for $s=0$, or $W_0=\emptyset$, we do not impose the opposite  assumption here. Both these relaxations are made  to accommodate  the inductive proof below. Finally, note that we may drop the assumption that $H \nsubseteq \co(n)$, as it follows from the fact that $|W_1|\ge4$ (a comet cannot contain two edges  not containing the center). For a 3-graph $G$ and a vertex $v\in V(G)$, let $G(v)$ denote the \emph{link graph} of $v$ in $G$, that is, the set of pairs of vertices which together with $v$ form an edge of $G$.
\begin{fact}\label{final}
If $H$ is an $n$-vertex, $n\ge15$, $\{P,C\}$-free 3-graph such that  $H\supset P_2\cup K_3$,  and, under the notation of Subsection \ref{prepa}, $z=|W_1|\ge4$ and $H_1\neq\emptyset$, then
	$$
	|H|< 4+\binom{n-5}{2} = |H_n|.
	$$
\end{fact}
\begin{proof}
	The proof is by induction on $n$ with the initial step $n=15$ done earlier. Let $n\ge16$.
	It can be easily checked that, since $H$ is $P$-free, for every $v\in W$ either
	\begin{equation}\label{FORF}
	F^0(v)=\emptyset\quad\mbox{ or }\quad F^2(v)=\emptyset.
	\end{equation}
	Moreover, by the definitions of $F^1$ and $F^2$,
	\begin{equation}\label{4and2}
	|F^1(v)|\le 4\quad{and }\quad|F^2(v)|\le 2.
	\end{equation}

If $W_0=\emptyset$, then we are done by an earlier proof (at the beginning of this section). Otherwise,
	fix $v\in W_0$ and observe that, by the remark preceding (\ref{HHH}), $|F^0(v)|\le |W_0|-1$. Thus, by (\ref{FORF}) and (\ref{4and2}), since $|W_0|=n-5-|W_1|\le n-9$,
	\begin{equation}\label{hv}
	|H(v)|=|F^0(v)|+ |F^1(v)|+|F^2(v)|\le 4+  \max\{2,|W_0|-1\}\le 4+n-10=n-6.
	\end{equation}

  Notice that $H-v$ satisfies the assumptions of Fact \ref{final}. Indeed, as the removal of $v$ affects $H_0$ only, in the obtained sub-3-graph we still have both, $H_1\neq \emptyset$ and $|W_1|\ge4$. Consequently, by the induction's assumption and (\ref{hv})

 $$
 |H|=|H-v|+|H(v)|< 4+\binom{n-6}{2} + n-6=4+\binom{n-5}{2} = |H_n|.
 $$
\end{proof}

\end{document}